\documentclass[11pt,a4paper]{article}
\usepackage{amsfonts,amssymb,amsmath,amscd,latexsym,makeidx,theorem,xcolor}

\author{Luigi Ambrosio\thanks{Supported by the ERC ADG Grant GeMeTneES} \\ \small
Scuola Normale Superiore, Pisa \\
l.ambrosio@sns.it
\and
Guido De Philippis \\ \small Scuola Normale
Superiore, Pisa \\
guido.dephilippis@sns.it \and
Luca Martinazzi\thanks{Supported by the Swiss National Fond Grant no. PBEZP2-129520.}\\
\small Centro De Giorgi, Pisa\\ luca.martinazzi@sns.it}

\date{July 21, 2010}

\title{Gamma-convergence of nonlocal perimeter functionals}

\newtheorem{trm}{Theorem}
\newtheorem{prop}[trm]{Proposition}
\newtheorem{cor}[trm]{Corollary}
\newtheorem{lemma}[trm]{Lemma}

\newcommand{\R}[1]{\mathbb{R}^{#1}}
\newcommand{\de}{\partial}
\newcommand{\ve}{\varepsilon}

\newcommand{\M}[1]{\mathcal{#1}}

\newenvironment{proof}{\noindent\emph{Proof.}}{\hfill$\square$\medskip}

\DeclareMathOperator{\loc}{loc}

\DeclareMathOperator*{\dist}{dist}
\DeclareMathOperator*{\Intm}{\int\!\!\!\!\!\! \rule[2.6pt]{6.5pt}{.4pt}}

\begin{document}
\maketitle


\section{Introduction}

For a measurable set $E\subset\R{n}$, $n\ge 1$, $0<s<1$, and a
connected open set $\Omega\Subset\R{n}$ with Lipschitz boundary (or
simply $\Omega=(a,b)\Subset\R{}$ if $n=1$), we consider the
functional
$$\M{J}_s(E, \Omega):=\M{J}^1_s(E,\Omega)+\M{J}^2_s(E, \Omega),$$
where
\begin{equation*}
\begin{split}
\M{J}^1_s(E, \Omega)&:=\int_{E\cap \Omega}\int_{E^c\cap \Omega}\frac{1}{|x-y|^{n+s}}dxdy,\\
\M{J}^2_s(E, \Omega)&:=\int_{E\cap \Omega}\int_{E^c\cap
\Omega^c}\frac{1}{|x-y|^{n+s}}dxdy+ \int_{E\cap
\Omega^c}\int_{E^c\cap \Omega}\frac{1}{|x-y|^{n+s}}dxdy.
\end{split}
\end{equation*}

The functional $\M{J}_s(E,\Omega)$ can be thought of as a fractional
perimeter of $E$ in $\Omega$ which is non-local in the sense that it
is not determined by the behaviour of $E$ in a neighbourhood of $\de
E\cap \Omega$, and which can be finite even if the Hausdorff
dimension of $\de E$ is $n-s>n-1$. Notice that the term
$\M{J}^1_s(E,\Omega)$ is simply half of the fractional Sobolev space
seminorm $|\chi_E|_{W^{s,1}(\Omega)}$, where $\chi_E$ denotes the
characteristic function of $E$. Roughly speaking this term
represents the $(n-s)$-dimensional fractional perimeter of $E$
\emph{inside} $\Omega$, while $\M{J}^2_s$ is the contribution near
$\de\Omega$. This can be made precise when letting $s\uparrow 1$. We
also recall the following elementary scaling property:
\begin{equation}\label{scale}
\M{J}^i_s(\lambda
E,\lambda\Omega)=\lambda^{n-s}\M{J}^i_s(E,\Omega)\qquad \text{for
}\lambda>0, \; i=1,2.
\end{equation}
This functional has already been investigated by several authors. In
\cite{V} Visintin studied some basic properties of $\M{J}_s$, and in
particular he showed that $\M{J}_s$ satisfies a suitable co-area formula, see
Lemma \ref{coarea} below. Caffarelli, Roquejoffre and Savin
\cite{CRS} studied the behavior of minimizers of $\M{J}_s$, proving
that if $E$ is a local minimizer of $\M{J}_s(\cdot,\Omega)$, i.e.
$$\M{J}_s(E,\Omega)\le \M{J}_s(F,\Omega)\quad \text{whenever
}E\Delta F\Subset\Omega,$$ then $(\de E)\cap \Omega$ is of class
$C^{1,\alpha}$ up to a set of Hausdorff codimension in $\R{n}$ at
least $2$.

As it is well-known (see for instance \cite{giu} and the references
therein), for minimizers $E$ of the classical De Giorgi's perimeter,
which we shall denote $P(E,\Omega)$, the regularity results are
stronger. The boundary of a local minimizer $E$ of $P(\cdot,\Omega)$
is analytic if $n\le 7$, it has (at most) isolated singularities
when $n=8$ and it is analytic up to a set of codimension at least
$8$ in $\R{n}$ if $n\ge 9$. This suggests that the results of
\cite{CRS} might not be optimal for $s$ close to 1. Motivated by
this, Caffarelli and Valdinoci \cite{CV} studied the limiting
properties of minimal sets for the $s$-perimeter as $s\to 1^-$.


Partly motivated by their work, we make a complete analysis of the
limiting properties, in the sense of $\Gamma$-convergence, of
$\M{J}_s$ as $s\to 1^-$, under no other assumption than the
measurability of the sets considered. Our proofs differ in
particular from those in \cite{CV} because they do not rely on
uniform (as $s\to 1^-$) regularity estimates on $s$-minimal
boundaries borrowed from \cite{CRS}. The only result we need from
\cite{CRS}, in the proof of our Lemma~\ref{star}, is the local
minimality of halfspaces, whose proof is reproduced in the appendix.

We start by proving a coercivity result.

\begin{trm}[Equi-coercivity]\label{trm1}
Assume that $s_i\uparrow 1$ and that $E_i$ are measurable sets
satisfying
$$\sup_{i\in\mathbb{N}}(1-s_i)\M{J}^1_{s_i}(E_i,\Omega')<\infty\qquad\forall
\Omega'\Subset\Omega.$$ Then $(E_i)$ is relatively compact in
$L^1_{\loc}(\Omega)$, any limit point $E$ has locally finite
perimeter in $\Omega$.
\end{trm}

Notice the scaling factor $(1-s)$, which accounts for the fact that $\M{J}^1_1(E,\Omega)=+\infty$
unless $E\subset \Omega^c$, or $\Omega\subset E$, as already shown by Br\'ezis \cite{Bre}.

Let $\omega_k$ denote the volume of the unit ball in $\R{k}$ for $k\ge 1$, and set $\omega_0:=1$.

\begin{trm}[$\Gamma$-convergence]\label{trm2}
For every measurable set $E\subset\R{n}$ we have
\begin{equation}\label{Gamma}
\begin{split}
\Gamma-\liminf_{s\uparrow 1}(1-s)\M{J}^1_{s}(E,\Omega)&\geq \omega_{n-1}P(E,\Omega),\\
\Gamma-\limsup_{s\uparrow 1}(1-s)\M{J}_{s}(E,\Omega)&\leq \omega_{n-1} P(E,\Omega),
\end{split}
\end{equation}
with respect to the local convergence in measure, i.e. the $L^1_{\rm
loc}$ convergence of the corresponding characteristic functions in
$\R{n}$.
\end{trm}

We recall that \eqref{Gamma} means that
$$\liminf_{i\to \infty}(1-s_i)\M{J}^1_{s_i}(E_i,\Omega)\geq \omega_{n-1}P(E,\Omega)\qquad \text{whenever } \chi_{E_i}\to \chi_E \text{ in } L^1_{\loc}(\R{n}),\;s_i\uparrow 1,$$
and that for every measurable set $E$ and sequence $s_i\uparrow 1$ there exists a sequence $E_i$ with $\chi_{E_i}\to \chi_E$ in $L^1_{\loc}(\R{n})$ such that
$$\limsup_{i\to\infty}(1-s_i)\M{J}_{s_i}(E_i,\Omega)\leq \omega_{n-1} P(E,\Omega).$$

We finally show that as $s\uparrow 1$ local minimizers converge to
local minimizers, where by a local minimizer of
$\M{J}_s(\cdot,\Omega)$ we mean a Borel set $E\subset\R{n}$ such
that $\M{J}_s(E,\Omega)\le \M{J}_s(F,\Omega)$ whenever $E\Delta
F\Subset\Omega$. Notice that if $E$ is a local minimizer of
$\M{J}_s(\cdot ,\Omega)$ and $\Omega'\subset\Omega$, then $E$ is
also a local minimizer of $\M{J}_s(\cdot,\Omega')$. A similar
definition holds for $P(\cdot,\Omega)$.

\begin{trm}[Convergence of local minimizers]\label{trm3}
Assume that $s_i\uparrow 1$, $E_i$ are local minimizer of
$\M{J}_{s_i}(\cdot, \Omega)$, and $\chi_{E_i}\to\chi_E$ in
$L^1_{\loc}(\R{n})$. Then
\begin{equation}\label{energybou}
\limsup_{i\to \infty}(1-s_i)\M{J}_{s_i}(E_i,\Omega')< + \infty \qquad \forall\Omega'\Subset\Omega,
\end{equation}
$E$ is a local minimizer of $P(\cdot,\Omega)$ and
$(1-s_i)\M{J}_{s_i}(E_i,\Omega')\to \omega_{n-1}P(E,\Omega')$ whenever
$\Omega'\Subset\Omega$ and $P(E,\partial\Omega')=0$.
\end{trm}

We point out that $\Gamma$-convergence results for functionals
reminiscent of $\M{J}^1_s(\cdot,\R{n})$ have been proven in
\cite{ngu1}, \cite{ngu2}.
\medskip

We fix some notation used throughout the paper:

\noindent -- we write $x\in\R{n}$ as $(x',x_n)$ with $x'\in\R{n-1}$
and $x_n\in\R{}$;

\noindent -- we denote by $H$ the halfspace $\{x:\ x_n\leq 0\}$ and
by $Q=(-1/2,1/2)^n$ the canonical unit cube;

\noindent -- we denote by $B_r(x)$ the ball of radius $r$ centered at $x$ and,
unless otherwise specified, $B_r:= B_r(0)$.

\noindent -- for every $h\in \R{n}$ and function $u$ defined on
$U\subset \R{n}$ we set $\tau_h u(x):= u(x+h)$ for all $x\in U-h$.

For the definition and basic properties of the perimeter $P(E,\Omega)$ in the
sense of De Giorgi we refer to the monographs \cite{AFP} and \cite{giu}.

\section{Proof of Theorem~\ref{trm1}}

The proof is a direct consequence of the Frechet-Kolmogorov
compactness criterion in $L^p_{{\loc}}$ (applied with $p=1$),
ensuring pre-compactness of any family $\M{G}\subset
L^1_{\loc}(\Omega)$ satisfying
$$
\lim_{h\to 0}\sup_{u\in\M{G}}\|\tau_hu-u\|_{L^1(\Omega')}=0
\qquad\forall\Omega'\Subset\Omega,
$$
and of the following pointwise upper bound on $\|\tau_hu-u\|_{L^1}$:
for all $u\in L^1(\Omega)$, $A\Subset\Omega$, $h\in \R{n}$ with
$|h|<\dist(A,\partial\Omega)/2$ and $s\in (0,1)$ we have
\begin{equation}\label{upper}
\|\tau_h u-u\|_{L^1(A)}\leq C(n)|h|^s(1-s) \M{F}_s(u,\Omega),
\end{equation}
where
\begin{equation}\label{F}
\M{F}_s(u,\Omega):=\int_\Omega\int_\Omega\frac{|u(x)-u(y)|}{|x-y|^{n+s}}dxdy.
\end{equation}
The functional $\M{F}_s$ is obviously related to $\M{J}_s^1$ by
$$
\M{F}_s(\chi_E,\Omega)=2\M{J}^1_s(E,\Omega).
$$
The upper bound \eqref{upper} is a
direct consequence of Proposition~\ref{pmazya} below, whose proof
can be found in \cite{Ma}. Since the inequality is not explicitly
stated in \cite{Ma}, we repeat
it for the reader's convenience.

\begin{prop} \label{pmazya} For all $u\in L^1(\Omega)$, $A\Subset\Omega$ and
$s\in (0,1)$ we have
\begin{equation}\label{main}
 \frac{\|\tau_h u-u\|_{L^1(A)}}{|h|^{s}}
\le C(n) (1-s) \int_{B_{|h|}} \frac{\|\tau_\xi u- u \|_{L^1(A_{|h|}
)}}{|\xi|^{n+s}} d\xi
\end{equation}
whenever $0<|h|<\dist(A,\partial\Omega)/2$, and
$A_{|h|}:=\{x\in\R{n}:\dist(x,A)<|h|\}$.
\end{prop}

We start with two preliminary results.

\begin{prop}\label{prop1}
Let $u\in L^1(\Omega)$, $h\in\R{n}$ and $A\Subset\Omega$ open with
$|h|<\dist(A,\partial\Omega)/2$. Then for any $z\in (0, |h|]$ we
have:
\begin{equation}\label{mainest}
\|\tau_h u-u \|_{L^1(A)} \le C(n)\frac { |h|}{z^{n+1}} \int_{B_z}
\|\tau_\xi u-u\|_{L^1(A_{|h|})}d\xi,
\end{equation}
where $A_{|h|}$ is as in Proposition~\ref{pmazya}.
\end{prop}

\begin{proof}
Fix a non-negative function $\varphi\in C^1_c(B_1)$  with $\int_{B_1}\varphi dx=1$. For $x\in
A$ and $z\in(0,|h|]$ we write
\begin{align*}
u(x)&=\frac 1 {z^n} \int_{B_z} u(x+y) \varphi\Big( \frac
{y}{z}\Big)dy
+\frac 1 {z^n} \int_{B_z} (u(x)-u(x+y)) \varphi\Big( \frac {y}{z}\Big)dy\\
&=: U(x,z)+V(x,z).
\end{align*}
Then we have
\begin{equation}\label{punt}
|u(x+h)-u(x)|\le |U(x+h,z)-U(x,z)|+|V(x+h,z)|+|V(x,z)|.
\end{equation}
The second and third terms can be easily estimated as follows:
\[
|V(x+h,z)|+|V(x,z)|\le\frac {\sup|\varphi|} {z^n} \int_{B_z}
\left\{|\tau_y u(x)-u(x)|+|\tau_y u(x+h)-u(x+h)|\right\}dy.
\]
For the first one instead notice that
\begin{align*}
\nabla_x U(x,z) &=  -
\frac 1 {z^{n+1}} \int_{B_z(x)} u(y) \nabla \varphi\Big( \frac {y-x}{z}\Big)dy\\
& = -\frac 1 {z^{n+1}} \int_{B_z(x)} (u(y)-u(x)) \nabla
\varphi\Big( \frac {y-x}{z}\Big)dy
\end{align*}
and so
\begin{align*}
|U(x+h,z)-U(x,z)|&\le  |h|\int_0^1 |\nabla_x U(x+sh,z)| ds \\
&\le\sup|\nabla\varphi|\frac {|h|} {z^{n+1}} \int_0^1 \int_{B_z}
|u(y+x+sh)-u(x+sh)|dyds.
\end{align*}
Notice now that $ z\le |h| $ and so $1\leq |h|/z$, hence from
\eqref{punt} we have:
\begin{eqnarray*}
|u(x+h)-u(x)|
&\le& C \Big\{ \frac 1 {z^n} \int_{B_z} |\tau_y u(x)-u(x)|+|\tau_y
u(x+h)-u(x+h)|dy\\
&&+\frac {|h|} {z^{n+1}} \int_0^1  \int_{B_z} |u(y+x+sh)-u(x+sh)|dy ds\Big\}\\
&\le& C \frac {|h|} {z^{n+1}} \Big\{  \int_{B_z} |\tau_y
u(x)-u(x)|+|\tau_y u(x+h)-u(x+h)|dy\\
&&+\int_0^1  \int_{B_z} |\tau_y u(x+sh)-u(x+sh)|dy ds\Big\},
\end{eqnarray*}
 with $C=\sup|\varphi|+\sup|\nabla\varphi|$.
 Integrating both sides over $A$ we infer \eqref{mainest} with $C(n)=3C$.
\end{proof}

Recall now the following version of Hardy's inequality:
\begin{prop}Let $g:\R{}\to [0,\infty)$ be a Borel
function, then for every $s>0$ we have
\begin{equation}\label{Hardy}
\int_0^r \frac 1 {\xi^{n+s+1}}\int_0^\xi g(t) dt d\xi \leq
\frac{1}{n+s}\int_0^r \frac {g(t)} {t^{n+s}} dt \qquad\forall r\geq
0.
\end{equation}
\end{prop}
 \begin{proof} We have
 \begin{align*}
& \int_0^r \frac 1 {\xi^{n+s+1}}\int_0^\xi g(t) dt d\xi
 =\int_0^r g(t)\int_t^r \frac 1 {\xi^{n+s+1}} d\xi dt\\
 &=\frac{1}{n+s}\int_0^r g(t)\Big(\frac{1}{t^{n+s}}-\frac{1}
 {r^{n+s}}\Big)dt\le\frac{1}{n+s}\int_0^r \frac{g(t)}{t^{n+s}} dt.
\end{align*}
\end{proof}

\noindent\emph{Proof of Proposition~\ref{pmazya}.} Multiply both sides of
\eqref{mainest} by $ z^{-s} $ and integrate with respect to $z$ between $ 0 $
and $ |h| $ to obtain
\[
\frac{|h|^{(1-s)}}{(1-s)}\|\tau_h u-u\|_{L^1(A)} \le
C(n)|h|\int_0^{|h|} \frac { 1}{z^{n+s+1}} \int_{B_z} \|\tau_\xi
u-u\|_{L^1(A_{|h|})}d\xi dz.
\]
 Now apply inequality \eqref{Hardy} with
 \[ g(t):=\int_{\partial B_t}  \|\tau_\xi u-u\|_{L^1(A_{|h|})} d\mathcal H^{n-1}(\xi)
 \]
 and obtain
 \begin{multline}
 \int_0^{|h|} \frac { 1}{z^{n+s+1}}
 \int_{B_z} \|\tau_\xi u-u\|_{L^1(A_{|h|})}d\xi dz
 =\int_0^{|h|} \frac { 1}{z^{n+s+1}} \int_0^z g(t) dt dz\\
 \leq C(n)\int_0^{|h|} \frac { 1}{t^{n+s}}  g(t) dt
=C(n)\int_{B_{|h|}} \frac { \|\tau_\xi u
-u\|_{L^1(A_{|h|})}}{|\xi|^{n+s}} d\xi.
 \end{multline}
 Putting all together
 \[
 \frac{\|\tau_h u-u\|_{L^1(A)}}{(1-s)}
 \le C(n)|h|^{s}\int_{B_{|h|}} \frac { \|\tau_\xi u -u\|_{L^1(A_{|h|})}}
 {|\xi|^{n+s}} d\xi
 \]
 and the thesis follows.
\hfill $\square$

\section{Proof of Theorem \ref{trm2}}

In the proof of the $\liminf$ inequality we shall adapt to this
framework the blow-up technique introduced, for the first time in
the context of lower semicontinuity, by Fonseca and M\"uller in
\cite{FM}. The proof of the $\limsup$ inequality, which is typically
constructive and by density, is slightly different from the
analogous results in \cite{CV}, since we approximate with polyhedra,
rather than $C^{1,\alpha}$ sets. Notice also that the natural
strategies in the proof of the $\liminf$ and $\limsup$ inequalities
produce constants $\Gamma_n$, see \eqref{defGamman}, and
$\Gamma_n^*\geq\Gamma_n$, see \eqref{gamman*}; our final task will
be to show that they both coincide with $\omega_{n-1}$.

\subsection{The $\Gamma-\liminf$ inequality}

Let us define
\begin{equation}\label{defGamman}
\Gamma_n:=\inf\Big\{\liminf_{s\uparrow 1}(1-s)\M{J}^1_s(E_s,Q)\;\Big|\;
\chi_{E_s}\to \chi_H\text{ in }L^1(Q)\Big\}.
\end{equation}
We denote by ${\mathcal C}$ the family of all $n$-cubes in $\R{n}$
$$
{\mathcal C}:=\left\{R(x+rQ):\ x\in\R{n},\,\,r>0,\,\,R\in
SO(n)\right\}.
$$

\begin{lemma}\label{lemmaliminf} Given $s_i\uparrow 1$ and sets $E_i\subset\R{n}$ with $\chi_{E_i}\to \chi_E$
in $L^1_{\loc}(\R{n})$ as $i\to \infty$, one has
\begin{equation}\label{liminf}
\liminf_{i\to\infty}(1-s_i)\M{J}^1_{s_i}(E_i,\Omega)\ge \Gamma_n P(E,\Omega).
\end{equation}
\end{lemma}

We can assume that the left-hand side of \eqref{liminf} is finite,
otherwise the inequality is trivial. Then, passing to the limit as
$i\to\infty$ in \eqref{main} with $s=s_i$ we get
$$
\|\tau_h\chi_E-\chi_E\|_{L^1(\Omega')}\leq C(n)|h|
\liminf_{i\to\infty}(1-s_i)\M{J}^1_{s_i}(E_i,\Omega)
\qquad\forall\Omega'\Subset\Omega
$$
whenever $|h|<\dist(\Omega',\partial\Omega)/2$, hence $E$ has finite
perimeter in $\Omega$.

 We shall
denote by $\mu$ the perimeter measure of $E$, i.e.
$\mu(A)=|D\chi_E|(A)$ for any Borel set $A\subset\Omega$, and we
shall use the following property of sets of finite perimeter: for
$\mu$-a.e. $x\in\Omega$ there exists $R_x\in SO(n)$ such that
$(E-x)/r$ locally converge in measure to $R_xH$ as $r\to 0$.
In addition,
\begin{equation}\label{densitycube}
\lim_{r\to 0}\frac{\mu(x+rR_xQ)}{r^{n-1}}=1,\quad \text{for  $\mu$-a.e. } x.
\end{equation}
Indeed this property holds for every $x\in \M{F} E$, where $\M{F}E$ denotes the reduced boundary of $E$, see Theorem 3.59(b) in \cite{AFP}.

Now, given a cube $C\in{\mathcal C}$ contained in $\Omega$ we set
$$\alpha_i(C):=(1-s_i)\M{J}^1_{s_i}(E_i,C)$$
and
$$\alpha (C):=\liminf_{i\to\infty}\alpha_i(C).$$
We claim that, setting $C_r(x):=x+rR_xQ$, where $R_x$ is as in
\eqref{densitycube}, for $\mu$-a.e. $x$ we have
\begin{equation}\label{eq1}
\liminf_{r\to 0}\frac{\alpha(C_r(x))}{\mu(C_r(x))}\geq \Gamma_n\quad
\text{for $\mu$-a.e. }x\in \R{n}.
\end{equation}
Then observing that for all $\ve>0$ the family
$$\M{A}:=\Big\{C_r(x)\subset\Omega\;:\;(1+\ve)\alpha(C_r(x))\geq \Gamma_n\mu(C_r(x))   \Big\}$$
is a fine covering of $\mu$-almost all of $\Omega$, by a suitable
variant of Vitali's theorem (see \cite{M}) we can extract a
countable subfamily of disjoint cubes
$$\{C_j\subset\Omega:j\in J\}$$ such that
$\mu\big(\Omega\setminus\bigcup\limits_{j\in J} C_j\big)=0$, whence
\begin{equation*}
\begin{split}
\Gamma_n P(E,\Omega)&=\Gamma_n\mu\Big(\bigcup_{j\in J}C_j\Big)=\Gamma_n\sum_{j\in J}\mu(C_j)\\
&\leq (1+\ve)\sum_{j\in J}\alpha(C_j)\le (1+\ve)\liminf_{i\to\infty}\sum_{j\in J}\alpha_i(C_j)\\
&\leq (1+\ve)\liminf_{i\to\infty}(1-s_i)\M{J}^1_{s_i}(E_i,\Omega).
\end{split}
\end{equation*}
In the last inequality we used that $\M{J}^1_s$ is superadditive and
positive for every $s\in (0,1)$. Since $\ve>0$ is arbitrary we get
the $\Gamma-\liminf$ estimate.

We now prove the inequality in \eqref{eq1} at any point $x$ such
that $(E-x)/r$ converges locally in measure as $r\to 0$ to
$R_xH$ and \eqref{densitycube} holds. Because of
\eqref{densitycube}, we need to show that
\begin{equation}\label{eq11}
\liminf_{r\to 0}\frac{\alpha(C_r(x))}{r^{n-1}}\geq \Gamma_n.
\end{equation}
Since from now on $x$ is fixed, we can assume with no loss of
generality (by rotation invariance) that $R_x=I$, so that the limit
hyperplane is $H$ and the cubes $C_r(x)$ are the standard ones
$x+rQ$. Let us choose a sequence $r_k\to 0$ such that
$$\liminf_{r\to 0}\frac{\alpha(C_r(x))}{r^{n-1}}=
\lim_{k\to \infty} \frac{\alpha(C_{r_k}(x))}{r_k^{n-1}}.$$ For $k>0$
we can choose $i(k)$ so large that the following conditions hold:
\begin{equation*}
\left\{
\begin{split}
&\alpha_{i(k)}(C_{r_k}(x))\leq \alpha(C_{r_k}(x))+r_k^n,\\
&r_k^{1-s_{i(k)}}\ge 1-\frac{1}{k},\\
&\Intm_{C_{r_k}(x)}|\chi_{E_{i(k)}}-\chi_E|dx<\frac{1}{k}.
\end{split}
\right.
\end{equation*}
Then we infer
\begin{equation*}
\begin{split}
\frac{\alpha(C_{r_k}(x))}{r_k^{n-1}}&\ge \frac{\alpha_{i(k)}(C_{r_k}(x))}{r_k^{n-1}}-r_k\\
&=\frac{(1-s_{i(k)})\M{J}^1_{s_{i(k)}}((E_{{i(k)}}-x)/r_k,Q)r_k^{n-s_{i(k)}}}{r_k^{n-1}}-r_k\\
&\ge
\Big(1-\frac{1}{k}\Big)(1-s_{i(k)})\M{J}^1_{s_{i(k)}}((E_{{i(k)}}-x)/r_k,Q)-r_k,
\end{split}
\end{equation*}
i.e.
$$\lim_{k\to\infty} \frac{\alpha(C_{r_k}(x))}{r_k^{n-1}}\ge
\liminf_{k\to\infty}(1-s_{i(k)})\M{J}^1_{s_{i(k)}}((E_{{i(k)}}-x)/r_k,Q).$$
On the other hand we have
$$\lim_{k\to\infty}\int_{Q}|\chi_{(E_{{i(k)}}-x)/r_k}-\chi_{(E-x)/r_k}|dx=0,$$
and
$$\lim_{k\to\infty}\int_{Q}|\chi_{(E-x)/r_k}-\chi_H|dx=0.$$
It follows that $(E_{{i(k)}}-x)/r_k\to H$ in $L^1(Q)$. Recalling
the definition of $\Gamma_n$ we conclude the proof of \eqref{eq11}
and of Lemma~\ref{lemmaliminf}.

\subsection{The $\Gamma-\limsup$ inequality}

It is enough to prove the $\Gamma-\limsup$ inequality for a
collection $\M{B}$ of sets of finite perimeter which is dense in
energy, i.e. such that for every set $E$ of finite perimeter there
exists $E_k\in\M{B}$ with $\chi_{E_k}\to \chi_E$ in
$L^1_{\loc}(\R{n})$ as $k\to\infty$ and
$\limsup_kP(E_k,\Omega)=P(E,\Omega)$. Indeed, let $d$ be a distance
inducing the $L^1_{\loc}$ convergence and, for a set $E$ of finite
perimeter, let $E_k$ be as above. Given $s_k\uparrow 1$, we can find
sets $\hat{E}_{k}$ with $d(\chi_{\hat{E}_{k}},\chi_{E_k})<1/k$ and
$$(1-s_k)\M{J}_{s_k}(\hat{E}_{k},\Omega)\leq \Gamma_n^* P(E_k,\Omega)+\frac{1}{k}.$$
Then we have $\chi_{\hat{E}_{k}}\to \chi_E$ in $L^1_{\loc}(\R{n})$
and
$$\limsup_{k\to\infty}(1-s_k)\M{J}_{s_k}(\hat{E}_{k},\Omega)
\leq \limsup_{k\to\infty}\Gamma_n^* P(E_k,\Omega)=\Gamma_n^*
P(E,\Omega).$$

We shall take $\M{B}$ to be the collection of polyhedra $\Pi$ which
satisfy $P(\Pi,\partial\Omega)=0$ (i.e. with faces transversal to
$\partial\Omega$, see Proposition \ref{poli}). Equivalently,
$$\lim_{\delta\to 0}P(\Pi,\Omega^+_\delta\cup \Omega^-_\delta)=0,$$
where
\begin{equation}\label{omegadelta}
\begin{split}
\Omega^+_\delta&:=\{x\in \Omega^c\;|\;d(x,\Omega)<\delta\}\\
\Omega^-_\delta&:=\{x\in \Omega\;|\;d(x,\Omega^c)<\delta\}.
\end{split}
\end{equation}

In fact, we have:

\begin{lemma}\label{poly} For a polyhedron $\Pi\subset\R{n}$ there holds
$$\limsup_{s\uparrow 1}(1-s)\M{J}_s(\Pi,\Omega)\leq
\Gamma_n^* P(\Pi,\Omega)+ 2\Gamma_n^* \lim_{\delta\to 0} P(\Pi,\Omega^+_\delta\cup \Omega^-_\delta),$$
where
\begin{equation}\label{gamman*}
\Gamma_n^*:= \limsup_{s\uparrow 1}(1-s)\M{J}^1_s(H,Q).
\end{equation}
\end{lemma}

\begin{proof} \emph{Step 1.} We first estimate $\M{J}_s^1(\Pi,\Omega)$. For a fixed $\ve>0$ set
\begin{equation*}
(\de\Pi)_\ve:=\{x\in\Omega\;|\; d(x,\de\Pi)<\ve\},\quad (\de\Pi)_\ve^-:=(\de\Pi)_\ve \cap \Pi.
\end{equation*}
We can find $N_\ve$ disjoint cubes $Q^\ve_i\subset\Omega$, $1\leq
i\le N_\ve$, of side length $\ve$ satisfying the following
properties:
\begin{itemize}
\item[(i)]
if $\tilde Q_i^\ve$ denotes the dilation of $Q_i^\ve$ by a factor
$(1+\ve)$, then each cube $\tilde Q_i^\ve$ intersects exactly one
face $\Sigma$ of $\de\Pi$, its barycenter belongs to $\Sigma$ and
each of its sides is either parallel or orthogonal to $\Sigma$;
\item[(ii)]  $\M{H}^{n-1}\left(((\de\Pi)\cap \Omega)\setminus \bigcup_{i=1}^{N_\ve}Q_i^\ve
\right)=|P(\Pi,\Omega)-N_\ve\ve^{n-1} |\to 0$ as $\ve \to 0$.
\end{itemize}
For $x\in\R{n}$ set
$$I_s(x):=\int_{\Pi^c\cap \Omega}\frac{dy}{|x-y|^{n+s}}.$$
We consider several cases.

\medskip

\noindent\emph{Case 1:} $x\in (\Pi\cap \Omega)\setminus(\de\Pi)_\ve^-$. Then for $y\in \Pi^c\cap \Omega$ we have $|x-y|\geq \ve$, hence
$$I_s(x)\leq \int_{(B_\ve(x))^c} \frac{1}{|x-y|^{n+s}}dy=n\omega_n\int_\ve^\infty \frac{1}{\rho^{s+1}}d\rho=\frac{n\omega_{n}}{s\ve^s},$$
since $n\omega_n=\M{H}^{n-1}(S^{n-1})$. Therefore
\begin{equation}\label{case1}
\int_{(\Pi\cap \Omega)\setminus(\de\Pi)_\ve^-}I_s(x)dx\leq \frac{n\omega_{n}\M{L}^n(\Pi\cap \Omega)}{s\ve^s}.
\end{equation}

\noindent\emph{Case 2:} $x\in
(\de\Pi)_\ve^-\setminus\bigcup_{i=1}^{N_\ve}Q_i^\ve$. Then
\begin{equation}\label{case2.0}
I_s(x)\leq\int_{(B_{d(x,\Pi^c \cap
\Omega)}(x))^c}\frac{1}{|x-y|^{n+s}}dy=
n\omega_{n}\int_{d(x,\Pi^c\cap\Omega)}^\infty\frac{1}{\rho^{s-1}}d\rho=\frac{n\omega_{n}}{s[d(x,\Pi^c\cap\Omega)]^s}.
\end{equation}
Now write $(\de\Pi)\cap \Omega=\bigcup_{j=1}^J\Sigma_j$, where each $\Sigma_j$ is the intersection of a face of $\de\Pi$ with $\Omega$, and define
$$(\de\Pi)^-_{\ve,j}:=\{x\in (\de\Pi)^-_\ve:\dist(x,\Pi^c\cap\Omega)=\dist(x,\Sigma_j)\}.$$
Clearly $(\de\Pi)^-_\ve=\bigcup_{j=1}^J (\de\Pi)^-_{\ve,j}$. Moreover we have
$$(\de\Pi)^-_{\ve,j}\subset\{x+t\nu:x\in \Sigma_{\ve,j},\, t\in (0,\ve),\, \nu\text{ is the interior unit normal to }\Sigma_{\ve,j}\}, $$
and $\Sigma_{\ve,j}$ is the set of points $x$ belonging to the same hyperplane as $\Sigma_j$ and with $\dist(x,\Sigma_j)\le \ve$. Clearly $\M{H}^{n-1}(\Sigma_{\ve,j})\le \M{H}^{n-1}(\Sigma_j)+C\ve$ as $\ve\to 0$. Then from \eqref{case2.0} we infer
\begin{equation}\label{case2}
\begin{split}
\int_{(\de\Pi)_\ve^-\setminus\bigcup_{i=1}^{N_\ve}Q_i^\ve}I_s(x)dx&\leq \frac{n\omega_{n}}{s}\sum_{j=1}^J\int_{(\de\Pi)_{\ve,j}^-\setminus\bigcup_{i=1}^{N_\ve}Q_i^\ve}\frac{1}{[d(x,\Pi^c)]^s}dx\\
&\leq \frac{n\omega_{n}}{s}\sum_{j=1}^J\int_{(\de\Pi)_{\ve,j}^-\setminus\bigcup_{i=1}^{N_\ve}Q_i^\ve}\frac{1}{[d(x,\Sigma_{\ve,j})]^s}dx\\
&\leq\frac{n\omega_{n}}{s}\sum_{j=1}^J\int_{(\Sigma_{\ve,j})\setminus\bigcup_{i=1}^{N_\ve}Q_i^\ve}\bigg(\int_0^\ve \frac{dt}{t^s}\bigg)d\M{H}^{n-1}\\
&= \frac{n\omega_{n}\ve^{1-s}}{s(1-s)}\M{H}^{n-1}\left(\bigg(\bigcup_{j=1}^J\Sigma_{\ve,j}\bigg)\setminus\bigcup_{i=1}^{N_\ve}Q_i^\ve\right)=\frac{\ve^{1-s}o(1)}{s(1-s)},
\end{split}
\end{equation}
with error $o(1)\to 0$ as $\ve\to 0$ and independent of $s$.

\medskip

\noindent\emph{Case 3:} $x\in \Pi\cap \bigcup_{i=1}^{N_\ve}Q_i^\ve$. In this case we write
\begin{equation*}
\begin{split}
I_s(x)&=\int_{(\Pi^c\cap \Omega)\cap \{y:|x-y|\geq \ve^2\}}\frac{dy}{|x-y|^{n+s}} +\int_{(\Pi^c\cap \Omega)\cap \{y:|x-y|< \ve^2\}}\frac{dy}{|x-y|^{n+s}}\\
&=:I_s^1(x)+I_s^2(x).
\end{split}
\end{equation*}
Then, similar to the case 1,
$$I_s^1(x)\leq n\omega_{n}\int_{\ve^2}^\infty\frac{1}{\rho^{s+1}}d\rho=\frac{n\omega_{n}}{s\ve^{2s}},$$
hence (since all cubes are contained in $\Omega$)
\begin{equation}\label{case3.1}
\int_{\Pi\cap \bigcup_{i=1}^{N_\ve}Q_i^\ve}I_s^1(x)dx\leq
\frac{\M{L}^n(\Omega)n\omega_{n}}{s\ve^{2s}}.
\end{equation}
As for $I_s^2(x)$ observe that if $x\in Q_i^\ve$ and $|x-y|\leq
\ve^2$, then $y\in \tilde Q_i^\ve$, where $\tilde Q_i^\ve$ is the
cube obtained by dilating $Q_i^\ve$ by a factor $1+\ve$ (hence the
side length of $\tilde Q_i^\ve$ is $\ve+\ve^2$). Then
\begin{equation}\label{case3.2}
\begin{split}
\int_{\Pi\cap \bigcup_{i=1}^{N_\ve}Q_i^\ve}I_s^2(x)dx&\le\sum_{i=1}^{N_\ve}\int_{\Pi\cap Q_i^\ve}\int_{\Pi^c\cap \tilde Q_i^\ve}\frac{1}{|x-y|^{n+s}}dydx \le\sum_{i=1}^{N_\ve}\int_{\Pi\cap \tilde Q_i^\ve}\int_{\Pi^c\cap \tilde Q_i^\ve}\frac{1}{|x-y|^{n+s}}dydx\\
&=N_\ve \M{J}^1_s(H,(\ve+\ve^2) Q) = N_\ve(\ve+\ve^2)^{n-s}\M{J}^1_s(H,Q),
\end{split}
\end{equation}
where in the last identity we used the scaling property \eqref{scale}.
Keeping $\ve>0$ fixed, letting $s$ go to $1$ and putting \eqref{case1}-\eqref{case3.2} together we infer
\begin{equation*}
\begin{split}
\limsup_{s\uparrow 1}(1-s)\M{J}^1_s(\Pi,\Omega)& \leq o(1)+\limsup_{s\uparrow 1}(1-s)\M{J}^1_s(H,Q)N_\ve(\ve+\ve^2)^{n-1}\\
&=o(1)+\Gamma_n^*P(\Pi,\Omega),
\end{split}
\end{equation*}
with error $o(1)\to 0$ as $\ve\to 0$ uniformly in $s$. Since $\ve>0$
is arbitrary, we conclude
$$\limsup_{s\uparrow 1}(1-s)\M{J}^1_s(\Pi,\Omega) \leq \Gamma_n^*P(\Pi,\Omega).$$

\medskip

\noindent\emph{Step 2.} It now remains to estimate $\M{J}_s^2$. Let us start by considering the term
$$\int_{\Pi\cap \Omega}\int_{\Pi^c\cap \Omega^c}\frac{1}{|x-y|^{n+s}}dydx.$$

\noindent\emph{Case 1:} $x\in \Pi\cap (\Omega\setminus\Omega_\delta^-)$. Then for
$y\in \Pi^c\cap\Omega^c$ we have $|x-y|\ge \delta$, whence
$$I(x):=\int_{\Pi^c\cap \Omega^c}\frac{dy}{|x-y|^{n+s}}\leq
n\omega_{n}\int_\delta^\infty\frac{d\rho}{\rho^{1+s}}=\frac{n\omega_{n}}{s\delta^s}.$$

\noindent\emph{Case 2:} $x\in \Pi\cap \Omega_\delta^-$. In this
case, using the same argument of case $1$ for $y\in
\Pi^c\cap(\Omega^c\setminus\Omega_\delta^+)$, we have
\begin{equation*}
\begin{split}
I(x)&=\int_{\Pi^c\cap\Omega_\delta^+}\frac{dy}{|x-y|^{n+s}}+
\int_{\Pi^c\cap(\Omega^c\setminus\Omega_\delta^+)}\frac{dy}{|x-y|^{n+s}}\\
&\le
\int_{\Pi^c\cap\Omega_\delta^+}\frac{dy}{|x-y|^{n+s}}+\frac{n\omega_{n}}{s\delta^s}.
\end{split}
\end{equation*}
Therefore
\begin{equation*}
\begin{split}
\int_{\Pi\cap \Omega}\int_{\Pi^c\cap \Omega^c}\frac{dydx}
{|x-y|^{n+s}}&\le\frac{2n\omega_{n}|\Omega|}{s\delta^s}+
\int_{\Pi\cap\Omega_\delta^-}\int_{\Pi^c\cap\Omega_\delta^+}\frac{dydx}{|x-y|^{n+s}}\\
&\le\frac{2n\omega_{n}|\Omega|}{s\delta^s}+\int_{\Pi\cap(\Omega_\delta^-\cup \Omega_\delta^+)}
\int_{\Pi^c\cap(\Omega_\delta^-\cup \Omega_\delta^+)}\frac{dydx}{|x-y|^{n+s}}.
\end{split}
\end{equation*}
An obvious similar estimate can be obtained by swapping $\Pi$ and
$\Pi^c$, finally yielding
\begin{equation*}
\begin{split}
\M{J}_s^2(\Pi,\Omega)&\le \frac{4n\omega_{n}|\Omega|}{s\delta^s}+2 \int_{\Pi\cap(\Omega_\delta^-\cup \Omega_\delta^+)}\int_{\Pi^c\cap(\Omega_\delta^-\cup \Omega_\delta^+)}\frac{dydx}{|x-y|^{n+s}}\\
&=
\frac{4n\omega_{n}|\Omega|}{s\delta^s}+2\M{J}_s^1(\Pi,\Omega_\delta^-\cup
\Omega_\delta^+).
\end{split}
\end{equation*}
Using the result of step 1 we get
$$\limsup_{s\uparrow 1}(1-s)\M{J}_s^2(\Pi,\Omega)\le 2\Gamma_n^*P(\Pi,\Omega_\delta^-\cup \Omega_\delta^+).$$
Since $\delta>0$ is arbitrary, letting $\delta$ go to zero we conclude the proof of the lemma.

\end{proof}

\begin{lemma}[Characterization of $\Gamma_n^*$]
\label{lemma*} The limsup in \eqref{gamman*} is a limit and $\Gamma_n^*=\omega_{n-1}.$
\end{lemma}

\begin{proof} The proof is inspired from \cite[Lemma 11]{CV}. We shall actually prove a slightly
stronger statement. Set for $a>0$
$$Q_a:=\{x: |x_i|\leq 1/2 \text{ for }1\le i\le n-1,\; |x_n|\le a \}.$$
Then we show that
$$\lim_{s\uparrow 1}(1-s)\M{J}_s^1(H,Q_a)=\omega_{n-1},\qquad\forall a>0.$$
Let us first consider the case $n\ge 2$.
Fix $x\in Q_a\cap H$ and write as usual $x=(x',x_n)$, $y=(y',y_n)$.
We consider
\begin{equation*}
I_s(x):=\int_{Q_a\cap H^c}\frac{1}{|x-y|^{n+s}}dy=
\int_0^a\int_{Q_a\cap \de H}\frac{1}{|x-y|^{n+s}}dy'dy_n.
\end{equation*}
With the change of variable $z'={(y'-x')}/{|y_n-x_n|}$ and setting
$$ \Sigma(x, y_n)
:=\left\{z'\in \R{n-1}:\left|z'_i +\frac{x'_i}{|x_n-y_n|}\right|\le
\frac{1}{2|x_n-y_n|}\text{ for }1\le i\le n-1\right\},$$ we get
\begin{equation}\label{conto}
\begin{split}
I_s(x)&=\int_0^a\int_{\Sigma(x,y_n)}\frac{1}{|x_n-y_n|^{s+1}(1+|z'|^2)^{(n+s)/2}}dz'dy_n\\
&\le \int_{0}^a\frac{1}{|x_n-y_n|^{s+1}}dy_n\cdot\int_{\R{n-1}}\frac{1}{(1+|z'|^2)^{(n+s)/2}}dz'\\
&=\frac{(-x_n)^{-s}-(a-x_n)^{-s}}{s} \cdot (n-1)\omega_{n-1}\int_0^\infty\frac{\rho^{n-2}}{(1+\rho^2)^{(n+s)/2}}d\rho.
\end{split}
\end{equation}
Now integrating $I$ with respect to $x$, observing that $\M{H}^{n-1}(Q_a\cap\de H)=1$ and that
by dominated convergence one has
\begin{equation}\label{integrale}
\begin{split}
\lim_{s\uparrow 1} \int_0^\infty\frac{\rho^{n-2}}{(1+\rho^2)^{(n+s)/2}}d\rho
&=\int_0^\infty\frac{\rho^{n-2}}{(1+\rho^2)^{(n+1)/2}}d\rho\\
&=\bigg[\frac{\rho^{n-1}}{(n-1)(1+\rho^2)^{(n-1)/2}}\bigg]_0^\infty=\frac{1}{n-1},
\end{split}
\end{equation}
we get
\begin{equation*}
\begin{split}
\int_{H\cap Q_a}I_s(x)dx&\le \M{H}^{n-1}(Q_a\cap \de H)\sup_{x'\in Q_a\cap \de H}\int_{-a}^0 I_s(x',x_n)dx_n\\
&\le  \omega_{n-1}(1+o(1))\int_{-a}^0\frac{(-x_n)^{-s}-(a-x_n)^{-s}}{s}dx_n\\
&=\frac{\omega_{n-1}(1+o(1))a^{1-s}(2-2^{1-s})}{s(1-s)},
\end{split}
\end{equation*}
with error $o(1)\to 0$ as $s\uparrow 1$ dependent only on $s$. Therefore
\begin{equation}\label{limsup}
\limsup_{s\uparrow 1}(1-s)\M{J}_s^1(H,Q_a)=\limsup_{s\uparrow 1}(1-s)\int_{H\cap Q_a}I_s(x)dx\le \omega_{n-1}.
\end{equation}
Now observing that for $\ve$ small enough
\begin{equation}\label{xy}
|x_n|\le \ve^2,\quad |y_n|\le \ve^2,\quad  |x_i|\leq \frac{1}{2}-\ve\text{ for }1\leq i\leq n-1
\end{equation}
implies that $ B_{1/(2\ve)}(0)\subset \Sigma(x,y_n),$ similar to
\eqref{conto} we estimate
\begin{equation*}
\begin{split}
I_s(x)&\ge \int_0^{\ve^2}\int_{Q\cap \de H}\frac{1}{|x-y|^{n+s}}dy'dy_n\\
&\geq\int_0^{\ve^2}\int_{B_{1/(2\ve)}(0)}\frac{1}{|x_n-y_n|^{s+1}(1+|z'|^2)^{(n+s)/2}}dz'dy_n\\
&=\frac{(-x_n)^{-s}-(\ve^2-x_n)^{-s}}{s}\cdot (n-1)\omega_{n-1}\int_0^\frac{1}{2\ve}\frac{\rho^{n-2}}{(1+\rho^2)^{(n+s)/2}}d\rho,
\end{split}
\end{equation*}
whenever $x$ is as in \eqref{xy}. Integrating with respect to $x$ satisfying \eqref{xy} one has
\begin{equation*}
\begin{split}
\int_{H\cap Q_a}I_s(x)dx&\ge (1-2\ve)^{n-1}\int_{-\ve^2}^0\frac{(-x_n)^{-s}-(\ve^2-x_n)^{-s}}{s}dx_n\\
&\quad \times (n-1)\omega_{n-1}\int_0^\frac{1}{2\ve}\frac{\rho^{n-2}}{(1+\rho^2)^{(n+s)/2}}d\rho\\
&=\frac{(n-1)\omega_{n-1}(1-2\ve)^{n-1}
\ve^{2(1-s)}(2-2^{1-s})}{s(1-s)}\int_0^\frac{1}{2\ve}\frac{\rho^{n-2}}{(1+\rho^2)^{(n+s)/2}}d\rho.
\end{split}
\end{equation*}
Letting first $s\uparrow 1$  and then $\ve\to 0$ and using
\eqref{integrale} again we conclude
$$\liminf_{s\uparrow 1}(1-s)\M{J}_s^1(H,Q_a)\ge \omega_{n-1},$$
which together with \eqref{limsup} completes the proof when $n\ge 2$.

When $n=1$ one computes explicitly
$$\M{J}^1_s(H,Q_a)=\int_{-a}^0\int_0^a \frac{1}{|x-y|^{1+s}}dydx=\int_{-a}^0\frac{(-x)^{-s}-(a-x)^{-s}}{s}dx=\frac{a^{1-s}(2-2^{1-s})}{s(1-s)},$$
hence
$$\lim_{s\uparrow 1}(1-s)\M{J}^1_s(H,Q_a)=1=\omega_0.$$
\end{proof}

\subsection{Gluing construction and characterization of the geometric constants}

A key observation in \cite{V}, which we shall need, is that $\M{F}$ satisfies a
generalized coarea formula, namely
$\M{F}_s(u,\Omega)=\int_0^1\M{F}_s(\chi_{\{u>t\}},\Omega)\,dt$; we
reproduce here the simple proof of this fact and we state the result
in terms of $\M{J}_s$.

\begin{lemma}[Coarea formula]\label{coarea} For every measurable function $u:\Omega\to [0,1]$ we have
$$\frac12 \M{F}_s(u,\Omega)=\int_0^1\M{J}_s^1 (\{u>t\},\Omega)dt.$$
\end{lemma}

\begin{proof} Given $x,y\in\Omega$, the function $t\mapsto\chi_{\{u>t\}}(x)-\chi_{\{u>t\}}(y)$
takes its values in $\{-1,0,1\}$ and it is nonzero precisely in the
interval having $u(x)$ and $u(y)$ as extreme points, hence
$$|u(x)-u(y)|=\int_0^1 |\chi_{\{u>t\}}(x)-\chi_{\{u>t\}}(y)|dt.$$
Substituting into \eqref{F}, using Fubini's theorem and observing that
$$|\chi_{\{u>t\}}(x)-\chi_{\{u>t\}}(y)|=
\chi_{\{u>t\}}(x)\chi_{\Omega\setminus\{u>t\}}(y)+
\chi_{\Omega\setminus\{u>t\}}(x)\chi_{\{u>t\}}(y),$$
we infer
\begin{equation*}
\begin{split}
\M{F}_s(u,\Omega)&=\int_\Omega\int_\Omega\int_0^1
\frac{|\chi_{\{u>t\}}(x)-\chi_{\{u>t\}}(y)|}{|x-y|^{n+s}}dtdxdy\\
&=2\int_0^1\int_{\{u>t\}}\int_{\Omega\setminus\{u>t\}}\frac{1}{|x-y|^{n+s}}dxdydt\\
&=2\int_0^1\M{J}^1_s(\{u>t\},\Omega)dt.
\end{split}
\end{equation*}
\end{proof}

\begin{prop}[Gluing]\label{propraccordo} Given $s\in (0,1)$, measurable sets $E_1$, $E_2$ in $\R{n}$
with $\M{J}^1_s(E_i,\Omega)<\infty$ for $i=1,2$ and given
$\delta_1>\delta_2>0$ we can find a measurable set $F$ such that
\begin{itemize}
\item[(a)] $\|\chi_{F}- \chi_{E_1}\|_{L^1(\Omega)}\le \|\chi_{E_1}-\chi_{E_2}\|_{L^1(\Omega)}$,
\item[(b)] $F\cap (\Omega\setminus\Omega_{\delta_1})=
E_1\cap (\Omega\setminus\Omega_{\delta_1})$, $F\cap \Omega_{\delta_2}=E_2 \cap \Omega_{\delta_2}$, where
$$\Omega_\delta:=\{x\in \Omega:d(x,\Omega^c)\le\delta\} \quad \text{for } \delta>0,$$
\item[(c)] for all $\ve>0$ we have
\begin{equation*}
\begin{split}
\M{J}^1_s(F,\Omega)\le&
\M{J}^1_s(E_1,\Omega)+\M{J}^1_s(E_2,\Omega_{\delta_1+\ve})+\frac{C}{\ve^{n+s}}\\
& +C(\Omega,\delta_1,\delta_2)\biggl[
\frac{\|\chi_{E_1}-\chi_{E_2}\|_{L^1(\Omega_{\delta_1}\setminus\Omega_{\delta_2})}}{(1-s)}
+\|\chi_{E_1}-\chi_{E_2}\|_{L^1(\Omega)}\biggr].
\end{split}
\end{equation*}
\end{itemize}
\end{prop}

\begin{proof} Consider a function $\varphi\in C^\infty(\R{n})$
such that $0\le \varphi\le 1$ in $\Omega$, $\varphi\equiv 0$ in
$\Omega_{\delta_2}$, $\varphi\equiv 1$ in
$\Omega\setminus\Omega_{\delta_1}$, and $|\nabla \varphi|\le
{2}/{(\delta_1-\delta_2)}$.

Given two measurable functions $u,\,v:\Omega\to [0,1]$ such that
$\M{F}_s(u,\Omega)<\infty$, $\M{F}_s(v,\Omega)<\infty$, define
$w:\Omega\to[0,1]$ as $w:=\varphi u+(1-\varphi)v$. For
$x,y\in\Omega$ we can write
\begin{equation*}
\begin{split}
w(x)-w(y)&=(\varphi(x)-\varphi(y))u(y)+\varphi(x)(u(x)-u(y))\\
&\quad+(1-\varphi(x))(v(x)-v(y)) -v(y)(\varphi(x)-\varphi(y))\\
&=(\varphi(x)-\varphi(y))(u(y)-v(y))+\varphi(x)(u(x)-u(y))\\
&\quad+(1-\varphi(x))(v(x)-v(y)),
\end{split}
\end{equation*}
and infer
\begin{equation*}
\begin{split}
|w(x)-w(y)|\le&|\varphi(x)-\varphi(y)| |u(y)-v(y)|\\
&+\chi_{\{\varphi\ne 0\}}(x)|u(x)-u(y)|+\chi_{\{\varphi\ne
1\}}(x)|v(x)-v(y)|.\\
\end{split}
\end{equation*}
Observing that $\{\varphi\neq 0\}\subset
\Omega\setminus\Omega_{\delta_2}$ and $\{\varphi\neq
1\}\subset\Omega_{\delta_1}$ we get
\begin{equation*}
\begin{split}
\M{F}_s(w,\Omega)&\le\int_\Omega|u(y)-v(y)|\int_\Omega\frac{|\varphi(x)-\varphi(y)|}{|x-y|^{n+s}}dxdy\\
&\quad +\int_{\Omega\setminus\Omega_{\delta_2}}\int_\Omega \frac{|u(x)-u(y)|}{|x-y|^{n+s}}dxdy+\int_{\Omega_{\delta_1}}\int_\Omega\frac{|v(x)-v(y)|}{|x-y|^{n+s}}dxdy\\
&=:I_1+I_2+I_3.
\end{split}
\end{equation*}
From
$$|\varphi(x)-\varphi(y)|\le |\nabla \varphi(y)||x-y|
+\frac{1}{2}\|\nabla^2\varphi\|_\infty|x-y|^2$$ and the inequalities
$\int_\Omega|x-y|^{-(n+s-\alpha)} dx\leq C(\Omega)/(\alpha-s)$ (with
$\alpha=1$, $\alpha=2$) we have
\begin{equation*}
\begin{split}
I_1&\le
\int_\Omega|u(y)-v(y)|\int_\Omega\bigg(\frac{|\nabla\varphi(y)|}
{|x-y|^{n+s-1}}+\frac{\|\nabla^2\varphi\|_\infty}{2|x-y|^{n+s-2}}\bigg)dxdy\\
&\le C(\Omega,\delta_1,\delta_2)\bigg(\frac{\|u-v\|_{L^1(\Omega_{\delta_1}\setminus\Omega_{\delta_2})}}{1-s}+\frac{\|u-v\|_{L^1(\Omega)}}{(2-s)}\Bigg).
\end{split}
\end{equation*}
Clearly $I_2\le \M{F}_s(u,\Omega).$ As for $I_3$, choosing $\ve>0$ we get
\begin{equation*}
\begin{split}
I_3&\le \int_{\Omega_{\delta_1}}\int_{\Omega_{\delta_1+\ve}}\frac{|v(x)-v(y)|}{|x-y|^{n+s}}dxdy
+\int_{\Omega_{\delta_1}}\int_{\Omega\setminus\Omega_{\delta_1+\ve}}\frac{|v(x)-v(y)|}{|x-y|^{n+s}}dxdy\\
&\le \M{F}_s(v,\Omega_{\delta_1+\ve})+\frac{2\M{L}^n(\Omega_{\delta_1})\M{L}^n(\Omega\setminus\Omega_{\delta_1+\ve})}{\ve^{n+s}}.
\end{split}
\end{equation*}
Summing up we obtain
\begin{equation}\label{stimetta}
\begin{split}
\M{F}_s(w,\Omega)\le& \M{F}_s(u,\Omega) + \M{F}_s(v,\Omega_{\delta_1+\ve}) +C(\Omega,\delta_1,\delta_2)\frac{\|u-v\|_{L^1(\Omega_{\delta_1}\setminus \Omega_{\delta_2})}}{1-s}\\
&+C(\Omega,\delta_1,\delta_2)\|u-v\|_{L^1(\Omega)}+\frac{C(\Omega)}{\ve^{n+s}}.
\end{split}
\end{equation}
We now apply this with $u=\chi_{E_1}$, $v=\chi_{E_2}$, so that
\eqref{stimetta} reads as
\begin{equation}\label{stimetta2}
\begin{split}
\M{F}_s(w,\Omega)\le& 2\M{J}^1_s(E_1,\Omega) + 2\M{J}^1_s(E_2,\Omega_{\delta_1+\ve}) +C(\Omega,\delta_1,\delta_2)\frac{\|\chi_{E_1}-\chi_{E_2}\|_{L^1(\Omega_{\delta_1}\setminus \Omega_{\delta_2})}}{1-s}\\
&+C(\Omega,\delta_1,\delta_2)\|\chi_{E_1}-\chi_{E_2}\|_{L^1(\Omega)}+\frac{C(\Omega)}{\ve^{n+s}},
\end{split}
\end{equation}
and by Lemma \ref{coarea} there exists $t\in (0,1)$ such that
$F:=\{w>t\}$ satisfies
$$2\M{J}^1_s(F,\Omega)\le \M{F}_s(w,\Omega).$$
By construction we see that $F$ satisfies conditions (a) and (b), and
by \eqref{stimetta2} it follows that also condition (c) is
satisfied.
\end{proof}

The following corollary is an immediate consequence of
Proposition~\ref{propraccordo}.

\begin{cor}\label{raccordo}Given measurable sets $E_s\subset \R{n}$ for $s\in(0,1)$,
with $\chi_{E_s}\to \chi_E$ in $L^1(\Omega)$ as $s\uparrow 1$ and with $\M{J}^1_s(E_s,\Omega)<\infty$, $\M{J}^1_s(E,\Omega)<\infty$, and given
$\delta_1>\delta_2>0$ we can find measurable sets $F_s\subset\R{n}$
such that
\begin{enumerate}
\item[(a)] $\chi_{F_s}\to \chi_E$ in $L^1(\Omega)$ as $i\to \infty$,
\item[(b)] $F_s\cap (\Omega\setminus\Omega_{\delta_1})= E_s\cap (\Omega\setminus\Omega_{\delta_1})$,
$F_s\cap \Omega_{\delta_2}=E \cap \Omega_{\delta_2}$,
\item[(c)] for all $\ve>0$ we have
\begin{equation*}
\liminf_{s\uparrow 1}(1-s)\M{J}^1_{s}(F_s,\Omega)\le
\liminf_{s\uparrow 1}(1-s)\M{J}^1_s(E_s,\Omega)+\limsup_{s\uparrow 1}(1-s)\M{J}^1_{s}(E,\Omega_{\delta_1+\ve}).
\end{equation*}
\end{enumerate}
\end{cor}

We devote the rest of the section to the proof of the equality of
the consants $\Gamma_n$ and $\Gamma_n^*$ appearing in the proof of the
$\Gamma$-liminf and $\Gamma$-limsup respectively (we already proved
that $\Gamma_n^*=\omega_{n-1}$). We shall introduce an intermediate
quantity $\tilde{\Gamma}_n\in [\Gamma_n,\Gamma_n^*]$ and prove in
two steps that $\tilde{\Gamma}_n=\Gamma_n$ (by the gluing
Proposition~\ref{propraccordo}) and then use the local minimality of
hyperplanes to show that $\tilde{\Gamma}_n=\Gamma_n^*$.

\begin{lemma}\label{tilde} We have $\Gamma_n=\tilde \Gamma_n$, where
$$\tilde \Gamma_n:=\inf\big\{\liminf_{s\uparrow 1}(1-s)\M{J}^1_s(E_s,Q)\big\},$$
with the infimum taken over all families of measurable sets
$(E_s)_{0<s<1}$ with the property that $\chi_{E_s}\to \chi_H \text{
in }L^1(Q)$ as $s\uparrow 1$ and, for some $\delta>0$, $E_s\cap Q^\delta=
H\cap Q^\delta$ for all $s\in (0,1)$, where $Q^\delta=\{x\in
Q:d(x,Q^c)<\delta\}$.
\end{lemma}

\begin{proof} Clearly $\tilde \Gamma_n\ge \Gamma_n$. In order to prove the converse consider sets
$E_s\subset\R{n}$ for $s\in (0,1)$ with $\chi_{E_s}\to \chi_H$ in
$L^1(Q)$ as $s\uparrow 1$. Without loss of generality we can assume
that $\M{J}^1_s(E_s,\Omega)<\infty$ for all $s\in (0,1)$. Then
according to Corollary~\ref{raccordo} for any given $\delta>0$ we
can find a family of measurable sets $(F_s)_{0<s<1}$ such that
$\chi_{F_s}\to \chi_H$ in $L^1(Q)$ as $s\uparrow 1$, $F_s\cap
Q^\delta=H\cap Q^\delta$ and
$$\tilde \Gamma_n\le \liminf_{s\uparrow 1}(1-s)\M{J}^1_s(F_s,\Omega)\le
\liminf_{s\uparrow 1}(1-s)\M{J}^1_s(E_s,Q)+\Gamma_n^*\inf_{\ve>0}
P(H,Q^{\delta+\ve}),$$
where we also used Lemma \ref{poly}.
Since $\delta>0$ is arbitrary and $P(H,Q^\delta)\rightarrow 0$ as $\delta\to 0$ we
infer
$$\tilde \Gamma_n\le \liminf_{s\uparrow 1}(1-s)\M{J}^1_s(E_s,Q)$$
and, since $(E_s)_{0<s<1}$ is arbitrary, this proves that $\tilde
\Gamma_n\le\Gamma_n$.
\end{proof}

\begin{lemma}\label{star} We have $\tilde \Gamma_n=\Gamma_n^*$.
\end{lemma}
\begin{proof} Clearly $\tilde \Gamma_n\le \Gamma_n^*$.
In order to prove the converse we consider sets $(E_s)_{0<s<1}$ with
$\chi_{E_s}\to \chi_H$ in $L^1(Q)$ as $s\uparrow 1$ and with $E_s\cap Q^\delta=H\cap
Q^\delta$ for some $\delta>0$ (here $Q^\delta$ is defined as in
Lemma~\ref{tilde}). Since our goal is to estimate $\M{J}^1_s(E_s,Q)$
from below, possibly modifying $E_s$ outside $Q$ we may assume that
\begin{equation}\label{EsH}
E_s\cap (\R{n}\setminus Q)= H \cap (\R{n}\setminus Q).
\end{equation}
This implies, according to Proposition \ref{Hmin} in the Appendix, that $\M{J}_s(H,Q)\le
\M{J}_s(E_s,Q)$ for $s\in (0,1)$. Then, in order to prove that
\begin{equation}\label{Js}
\lim_{s\uparrow 1}(1-s)\M{J}^1_s(H,Q)\le \liminf_{s\to
1^-}(1-s)\M{J}^1_s(E_s,Q),
\end{equation}
it is enough to show that
\begin{equation}\label{Js2}
\lim_{s\uparrow 1}(1-s)(\M{J}^2_s(H,Q)-\M{J}^2_s(E_s,Q))=0.
\end{equation}
One immediately sees that \eqref{EsH} imples
\begin{equation*}
\begin{split}
|\M{J}^2_s(H,Q)-\M{J}^2_s(E_s,Q)|\le& \int_{(E_s\Delta H)\cap Q}\int_{H^c\cap Q^c}\frac{1}{|x-y|^{n+s}}dxdy\\
&+ \int_{(E_s^c\Delta H^c)\cap Q}\int_{H\cap Q^c}\frac{1}{|x-y|^{n+s}}dxdy=:I+II.
\end{split}
\end{equation*}
Observing that $(E_s\Delta H)\cap Q^\delta=\emptyset$ we can estimate for $y\in(E_s\Delta H)\cap Q$
$$\int_{H^c\cap Q^c}\frac{1}{|x-y|^{n+s}}dx\le
\int_{\R{n}\setminus
B_\delta(y)}\frac{1}{|x-y|^{n+s}}dx=\frac{n\omega_{n}}{s\delta^s},
$$ hence $I\le {n\omega_{n}}/{(s\delta^s)}$. One can bound from above $II$
in the same way. Now \eqref{Js2} follows at once upon multiplying by
$1-s$ and letting $s\uparrow 1$. This shows \eqref{Js}, and taking the
infimum in \eqref{Js} over all families $(E_s)_{0<s<1}$ as above
shows that $\Gamma_n^*\le \tilde \Gamma_n$.
\end{proof}

\section{Proof of Theorem \ref{trm3}}

In order to prove \eqref{energybou} define $\Omega_\delta$ as in
Proposition~\ref{propraccordo} for some small $\delta>0$ and set
$F_i:= E_i\cap(\Omega^c\cup \Omega_\delta)$. By the minimality of
$E_i$ we then have
\begin{equation*}
\begin{split}
\limsup_{i\to\infty}(1-s_i)\M{J}_{s_i}(E_i,\Omega\setminus\Omega_\delta)&\le
\limsup_{i\to\infty}(1-s_i)\Big(\M{J}_{s_i}(E_i,\Omega)-\M{J}^1_{s_i}(E_i,\Omega_\delta)\Big)\\
&\le \limsup_{i\to\infty}(1-s_i)\Big(\M{J}_{s_i}(F_i,\Omega)-\M{J}^1_{s_i}(E_i,\Omega_\delta)\Big)\\
&= \limsup_{i\to\infty}(1-s_i)\Big[\big(\M{J}^1_{s_i}(F_i,\Omega)-\M{J}^1_{s_i}(F_i,\Omega_\delta) \big)+
\M{J}^2_{s_i}(F_i,\Omega)\Big].
\end{split}
\end{equation*}
Since $F_i\cap (\Omega\setminus\Omega_\delta)=\emptyset$ we have, using Proposition \ref{punt2} in the appendix,
\begin{equation*}
\begin{split}
\limsup_{i\to\infty}(1-s_i)\big(\M{J}^1_{s_i}(F_i,\Omega)-\M{J}^1_{s_i}(F_i,\Omega_\delta) \big)&\le
\limsup_{i\to\infty}(1-s_i)\M{J}^1_{s_i}(\Omega\setminus \Omega_\delta,\Omega)\\
&= \limsup_{i\to\infty}(1-s_i)\frac{\M{F}_{s_i}(\chi_{\Omega\setminus\Omega_\delta},\Omega)}{2}\\
&\le \frac{n\omega_n P(\Omega\setminus \Omega_\delta,\R{n})}{2}.
\end{split}
\end{equation*}
Again using Proposition \ref{punt2} in the appendix we get
$$\limsup_{i\to\infty}(1-s_i)\M{J}^2_{s_i}(F_i,\Omega)\le
\limsup_{i\to\infty}(1-s_i)\M{J}^1_{s_i}(\Omega,\R{n})\le \frac{n\omega_{n}P(\Omega,\R{n})}{2},$$
whence \eqref{energybou} follows for $\Omega'\subset\Omega\setminus\Omega_\delta$,
hence for every $\Omega'\Subset\Omega$.

For the sake of simplicity we first consider perturbations in
compactly supported balls. The general case will require only minor
modifications.

Consider the monotone set function $\alpha_i(A):=(1-s_i)\M{J}^1_{s_i}(E_i,A)$ for every open set $F\subset\Omega$ (see the appendix for the definition and some basic properties of monotone set functions), extended to
$$\alpha_i(F):=\inf\{\alpha_i(A): F\subset A\subset\Omega, \ A\text{ open}\}$$
for every $F\subset\Omega.$
Clearly $\alpha_i$ is regular. Thanks to \eqref{energybou} and Theorem \ref{DGL}, up to  extracting a subsequence, $\alpha_i$ weakly converges to a regular monotone set function $\alpha$, which is regular and super-additive.
We shall now prove that if $B_R(x)\Subset\Omega$ and $\alpha(\de B_R(x))=0$, then $E$ is a local minimum of
the functional $P(\cdot,B_R(x))$, and
$$\lim_{i\to\infty}(1-s_i)\M{J}_{s_i}(E_i,B_R(x))=P(E,B_R(x)).$$
Indeed consider a Borel set $F\subset\Omega$ such that $E\Delta
F\Subset B_R$ (here and in the following $x$ is fixed and
$B_r:=B_r(x)$ for any $r>0$). Then we can find $r<R$ such that
$E\Delta F\subset B_r$. By Theorem \ref{trm2} there exist sets $F_i$
such that
$$\lim_{i\to \infty}|(F_i\Delta F)\cap B_R|= 0,\quad \lim_{i\to \infty}
(1-s_i)\M{J}_{s_i}(F_i, B_R)=\omega_{n-1} P(F,B_R).$$ According to Proposition \ref{propraccordo}, given $\rho$ and $t$ with $r<\rho<t<R$, we can find sets
$G_i$ such that
$$G_i= E_i \text{ in } \R{n}\setminus B_t,\quad G_i=F_i\text{ in }B_\rho,$$
and for all $\ve>0$ there holds
\begin{equation*}
\begin{split}
\M{J}^1_{s_i}(G_i, B_R)&\le \M{J}^1_{s_i}(F_i, B_R)+\M{J}^1_{s_i}(E_i, B_R\setminus \overline B_{\rho-\ve})+\frac{C}{\ve^{n+s_i}}\\
&\quad + \frac{C|(E_i\Delta F_i)\cap (B_t\setminus
B_\rho)|}{(1-s_i)}+C|(F_i\Delta E_i)\cap B_R|.
\end{split}
\end{equation*}
By the local minimality of $E_i$ we infer
\begin{equation*}
\M{J}_{s_i}(E_i, B_R)\le \M{J}_{s_i}(G_i, B_R).
\end{equation*}
We shall now estimate
\begin{equation*}
\begin{split}
\M{J}_{s_i}^2(G_i,B_R)&=\int_{G_i\cap B_R}\int_{G_i^c\cap B_R^c}\frac{dxdy }{|x-y|^{n+s_i}}
+\int_{G_i^c\cap B_R}\int_{G_i\cap B_R^c}\frac{dxdy}{|x-y|^{n+s_i}}\\
&=:I+II
\end{split}
\end{equation*}
We have
\begin{equation*}
\begin{split}
I=&\int_{G_i\cap B_R}\int_{E_i^c\cap B_R^c}\frac{dxdy
}{|x-y|^{n+s_i}}
=\int_{G_i\cap B_t}\int_{E_i^c\cap B_R^c}\frac{dxdy }{|x-y|^{n+s_i}}+
\int_{E_i\cap (B_R\setminus B_t)}\int_{E_i^c\cap B_R^c}\frac{dxdy }{|x-y|^{n+s_i}}\\
\le&\frac{C|G_i\cap B_t|}{s_i(R-t)^{s_i}}+\int_{E_i\cap (B_R\setminus
B_t)} \int_{E_i^c\cap (B_{R'}\setminus B_R)}\frac{dxdy
}{|x-y|^{n+s_i}}+
\int_{E_i\cap (B_R\setminus B_t)}\int_{E_i^c\cap B_{R'}^c}\frac{dxdy }{|x-y|^{n+s_i}}\\
\le & \M{J}^1_{s_i}(E_i,B_{R'}\setminus \overline B_t)+
\frac{C}{s_i}\bigg(\frac{1}{(R-t)^{s_i}}+\frac{1}{(R'-R)^{s_i}}\bigg),
\end{split}
\end{equation*}
for any $R'\in (R,{\rm dist}(x,\partial\Omega))$. Since $II$ can be
estimated in a similar way, we infer
$$\M{J}_{s_i}^2(G_i,B_R)\le2\M{J}^1_{s_i}(E_i, B_{R'}\setminus \overline B_t)+
\frac{C}{s_i}\bigg(\frac{1}{(R-t)^{s_i}}+\frac{1}{(R'-R)^{s_i}}\bigg),
$$
hence,
$$
\limsup_{i\to\infty}(1-s_i)\M{J}^2_{s_i}(G_i,B_R)\le
2\limsup_{i\to\infty}(1-s_i)\M{J}^1_{s_i}(E_i, B_{R'}\setminus \overline B_t).$$
Finally
\begin{equation}\label{chain}
\begin{split}
\omega_{n-1}P(E,B_R)\le & \liminf_{i\to \infty}(1-s_i)\M{J}^1_{s_i}(E_i, B_R)\le  \liminf_{i\to \infty}(1-s_i)\M{J}_{s_i}(E_i,B_R)\\
\le & \liminf_{i\to \infty}(1-s_i)\M{J}_{s_i}(G_i,B_R)\\
\le & \liminf_{i\to \infty}(1-s_i)\M{J}^1_{s_i}(G_i,B_R)+\limsup_{i\to \infty}(1-s_i)\M{J}^2_{s_i}(G_i,B_R)\\
\le & \liminf_{i\to
\infty}(1-s_i)\M{J}^1_{s_i}(F_i,B_R)+3\limsup_{i\to \infty}
(1-s_i)\M{J}^1_{s_i}(E_i,B_{R'}\setminus \overline B_{\rho-\ve})\\
&+C\lim_{i\to\infty} |(E_i\Delta F_i)\cap (B_t\setminus B_\rho)|.
\end{split}
\end{equation}
The last term is zero, since $E=F$ in $B_t\setminus B_\rho$ and
$|(E_i\Delta E)\cap B_R|\to 0$, $|(F_i\Delta F)\cap B_R|\to 0$ as
$i\to\infty$. Using Proposition \ref{useful} from the appendix, and
recalling that $\alpha(\de B_R)=0,$ we infer
$$\lim_{R'\downarrow R,\ \rho\uparrow R,\ \ve\downarrow 0}
\limsup_{i\to \infty}(1-s_i)\M{J}^1_{s_i}(E_i,B_{R'}\setminus
\overline B_{\rho-\ve})=\lim_{\delta\to 0}\limsup_{i\to\infty}
\alpha_i(B_{R+\delta}\setminus \overline B_{R-\delta})=0,$$
and \eqref{chain} finally yields

$$\omega_{n-1} P(E,B_R)\le \lim_{i\to \infty}(1-s_i)\M{J}_{s_i}(F_i,B_R)=\omega_{n-1} P(F,B_R),$$
so $E$ is a local minimizer of $P(\cdot, B_R)$.
Choosing $F=E$ the chain of inequalities in \eqref{chain} gives
\begin{equation}\label{ener}
\lim_{i\to\infty} (1-s_i)\M{J}_{s_i}(E_i,B_R)=\lim_{i\to\infty} (1-s_i)\M{J}^1_{s_i}(E_i,B_R)=\omega_{n-1}P(E,B_R),
\end{equation}
as wished. In order to complete the proof we first remark that the above arguments applies to any open set $\Omega'\Subset\Omega$ with Lipschitz boundary and $\alpha(\de\Omega')=0$, upon replacing $B_R(x)$ by $\Omega'$,  $B_{R+\delta}$ by $N_\delta(\Omega')$ and $B_{R-\delta}$ by $N_{-\delta}(\Omega')$, where
$$N_\delta(\Omega'):=\{x\in \Omega: d(x,\Omega')<\delta\}, \quad N_{-\delta}(\Omega'):=\{x\in\Omega':d(x,\de\Omega')>\delta\}\quad \text{for $\delta>0$ small}.$$
In particular $\alpha(\Omega')=P(E,\Omega')$ for every open set $\Omega'\Subset\Omega$ with Lipschitz boundary and  $\alpha(\de\Omega')=0$. Since for every $\Omega'\Subset \Omega$ and $\ve>0$ small enough the set
$$\{\delta\in (-\ve,\ve):\alpha(\de N_\delta(\Omega'))>0\}$$
is at most countable (remember that $\alpha$ is super-additive and locally finite), and since both $\alpha$ and $P(E,\cdot)$ are \emph{regular} monotone set functions on $\Omega$, it is not difficult to show that $\alpha=P(E,\cdot)$, and the proof is complete.
\hfill $\square$

\section{Appendix. Some useful results}

We list here some results which we used in the previous sections.

\begin{prop}\label{poli} Let $ E\subset\R{n} $ be a set
with finite perimeter in $ \Omega $. Then for every $ \ve >0 $ there
exists a polyhedral set $ \Pi\subset\R{n}  $ such that
\begin{itemize}
\item[(i)] $|(E\triangle \Pi) \cap \Omega|< \ve,$
\item[(ii)]$ |P(E,\Omega)-P(\Pi,\Omega)|< \ve ,$
\item[(iii)]$P(\Pi,\de \Omega)=0.$
\end{itemize}
\end{prop}

\begin{proof}
Classical theorems (see for example \cite{AFP,DG}) imply that there exists a polyhedral set $ \Pi' $ satisfying (i) and (ii). In order to get (iii) first notice that
\[
P(\Pi',\de \Omega)>0\text{ if and only if } \mathcal H^{n-1}(\de \Pi' \cap \de \Omega)>0,
\]
and that the latter condition can be satisfied only if $ \de \Omega
$ contains a piece $ \Sigma $ with $\M{H}^{n-1}(\Sigma)>0$ contained
in a hyperplane and $ \nu_\Omega=\pm \nu_{\Pi'}=\rm const$ on
$\Sigma$ (here $\nu_\Omega$ and $\nu_{\Pi'}$ denote the interior
unit normal to $\de \Omega$ and $\de \Pi'$ respectively). Since the
set
$$\big\{\nu\in S^{n-1}: \M{H}^{n-1}(\{x\in\de\Omega:\nu_\Omega(x)=\nu\})>0\big\} $$
is at most countable, it is easy to see that there exists a rotation
 $R\in SO(n)$ close enough to the identity so that the polyhedron
 $ \Pi:=R(\Pi') $ satisfies (i), (ii) and (iii).
\end{proof}

\begin{prop}\label{punt2} Let $ u \in BV(\Omega)$  and let $\Omega'\Subset \Omega$ be open. Then we have
\begin{equation} \label{th}
\limsup_{s\uparrow 1} (1-s) \M{F}_s(u,\Omega')\leq
n\omega_{n}\limsup_{|h|\to 0}
\int_{\Omega'}\frac{|u(x+h)-u(x)|}{|h|}dx\le
n\omega_{n}|Du|(\Omega).
\end{equation}
\end{prop}

\begin{proof}
For $h\in\R{n}$ let us define
\[
g(h)= \int_{\Omega'}\frac{|u(x+h)-u(x)|}{|h|}dx
\]
and fix $ L>\limsup_{|h|\to 0} g(h) $. Then there exists
$\delta_L>0$ such that $\Omega'+h\subset\Omega$ for all $h\in
B_{\delta_L}$ and $L\ge g(h)$ for $0<|h|\le \delta_L$. Multiplying
by $|h|^{-n-s+1}$ and integrating with respect to $h$ on
$B_{\delta_L}$ we obtain
\begin{equation}\label{1}
\frac{n\omega_n \delta_L^{1-s}L}{1-s}\ge \int_{B_{\delta_L}} \frac {g(h)}{|h|^{n+s-1}}dh =\int_{B_{\delta_L}} \int_{\Omega'}\frac{|u(x+h)-u(x)|}{|h|^{n+s}}dxdh.
\end{equation}
Now notice that
\begin{equation}\label{2}
\begin{split}
&\int_{\Omega'}\int_{\Omega'}\frac{|u(x)-u(y)|}{|x-y|^{n+s}}dxdy \\
 =& \int_{(\Omega'\times \Omega')\cap \{|x-y|\le \delta_L\}}\frac{|u(x)-u(y)|}{|x-y|^{n+s}}dxdy+\int_{(\Omega'\times \Omega')\cap \{|x-y|\ge \delta_L\}}\frac{|u(x)-u(y)|}{|x-y|^{n+s}}dxdy\\
 \le& \int_{B_{\delta_L}} \int_{\Omega'}\frac{|u(x+h)-u(x)|}{|h|^{n+s}}dxdh+\int_{B_{\delta_L}^c} \int_{\Omega'}\frac{|u(x+h)-u(x)|}{|h|^{n+s}}dxdh\\
 \le& \int_{B_{\delta_L}} \int_{\Omega'}\frac{|u(x+h)-u(x)|}{|h|^{n+s}}dx + \frac { 2n\omega_n}{s\delta_L^s}\| u\|_{L^1(\Omega)}.
\end{split}
\end{equation}
 Putting together \eqref{1} and \eqref{2} we obtain
 \[
 n\omega_n L \ge \limsup_{s\uparrow 1} (1-s)  \int_{\Omega'}\int_{\Omega'}\frac{|u(x)-u(y)|}{|x-y|^{n+s}}dxdy,
 \]
 and for $ L\to \limsup\limits_{|h|\to 0} g(h) $ the first inequality in \eqref{th}. The second one is well-known.
\end{proof}

\subsection{Minimality of $H$}

\begin{prop}\label{Hmin} For every $s\in(0,1)$, $H$ is the unique
minimizer of $\M{J}_s(\cdot, Q)$, in the sense
that $\M{J}_s(H,Q)\le\M{J}_s(F,Q)$ for every set $F\subset\R{n}$
with $F\cap Q^c=H\cap Q^c$, with strict inequality if
$F\neq H$.
\end{prop}

The proof of Proposition \ref{Hmin} easily follows from a couple of results of \cite{CRS}, which we give here for the sake of completeness.

\begin{prop}[Existence of minimizers]\label{exist} Given $E_0\subset\Omega^c$ and $s\in(0,1)$ there exists $E\subset\R{n}$ such that $E\cap \Omega^c=E_0$ and
\begin{equation}\label{minimo}
\inf_{F\cap \Omega^c=E_0}\M{J}_s(F,\Omega)=\M{J}_s(E,\Omega).
\end{equation}
\end{prop}

\begin{proof}
This follows immediately from the lower semicontinuity of $\M{J}_s$ with respect to the $L^1_{\loc}$ convergence (a simple consequence of Fatou's lemma) and the coercivity estimate of Proposition \ref{pmazya}.
\end{proof}

In general a set $E$ satisfying \eqref{minimo} will be called a
\emph{minimizer} of $\M{J}_s(\cdot,\Omega)$. Following the notation
of \cite{CRS}, we set $L(A,B):=\int_A\int_B |x-y|^{-n-s}dxdy$ for
$s\in(0,1)$ and $A,\,B\subset\R{n}$ measurable. Notice that $L(A\cup
B,C)=L(A,C)+L(B,C)$ if $|A\cap B|=0$ and $L(A,B)=L(B,A)$. Now we can
write
$$\M{J}_s(E,\Omega)=L(E\cap \Omega, E^c)+L(E\cap \Omega^c, E^c\cap \Omega).$$
It is easy to check that a minimizer $E$ of
$\M{J}_s(\cdot,\Omega)$ satisfies
\begin{eqnarray}
L(A,E)\leq L(A,E^c\setminus A)&&\qquad \text{for }A\subset E^c\cap \Omega \label{sub}\\
L(A,E^c)\leq L(A,E\setminus A)&& \qquad \text{for }A\subset E\cap
\Omega \label{sup}.
\end{eqnarray}
It suffices indeed to compare $E$ with $E\setminus A$ and with $E\cup A$.

\begin{prop}[Comparison principle I]\label{comp} Let $E$ satisfy \eqref{sub} with $\Omega=Q$
and assume that $H\cap Q^c\subset E$. Then $H\subset E$ up to a set of measure zero (i.e. $|H\cap E^c|=0$).
\end{prop}

\begin{proof} Let $T(x',x_n):=(x',-x_n)$ denote the reflection across $\de H$ and set
$A^-:= H\cap E^c$, $A^+:= T(A^-)\cap E^c$, $A:=A^-\cup A^+\subset
E^c\cap Q$, $A_1:= A^+\cup T(A^+)$, $A_2=A^-\setminus T(A^+)$ and
$F:= T(E^c\setminus A)\subset H$. Then, observing that
$L(B,C)=L(T(B),T(C))$, from \eqref{sub} we infer
\begin{equation*}
\begin{split}
0&\ge L(A,E)-L(A,E^c\setminus A)=L(A,E)-L(T(A),F)= L(A,E)-L(A_1,F)-L(T(A_2),F)\\
&=L(A,E)-L(A,F)+L(A_2,F)-L(T(A_2),F)=L(A,E\setminus F)+L(A_2,F)-L(T(A_2),F)\\
&=L(A_1,E\setminus F)+L(A_2, E\setminus F)+(L(A_2,F)-L(T(A_2),F)).
\end{split}
\end{equation*}
The first two terms on the right-hand side are clearly positive. We
also have $L(A_2,F)>L(T(A_2),F)$ unless $|A_2|=0$, since for $y\in
F$ and $x\in A_2\setminus \de H$ one has $|x-y|<|T(x)-y|$. Therefore
the right-hand side must be zero, $|A_2|=0$ and either $|A_1|=0$
(and the proof is complete), or $|E\setminus F|=0$. In the latter
case consider for a small $\ve>0$ the translated set
$E_\ve:=E+(0,\ldots,0,\ve)$, which satisfies \eqref{sub} in
$Q_\ve:=Q+(0,\ldots,0,\ve)$, hence also in $\tilde Q_\ve:=Q_\ve\cap
T(Q_\ve)$. Repeating the above procedure for $E_\ve$ in $\tilde
Q_\ve$ we get $|A_{2,\ve}|=0$ ($A^-_\ve$, $A^+_\ve$, etc. are
defined as above with respect to the set $E_\ve$
 in the domain $\tilde Q_\ve$, still reflecting
across $\partial H$; we use also the fact since $H\subset
H_\ve:=H+(0,\ldots,0,\ve)$, we have $H\cap \tilde Q_\ve^c\subset
E_\ve$) and, since $|E_\ve\setminus F_\ve|=\infty$, $|A_{1,\ve}|=0$.
This implies at once that $|A^-_\ve|=0$ and $|H\setminus E_\ve|=0$.
Since this is true for every small $\ve>0$, it follows that
$H\subset E$ (up to a set of measure $0$).
\end{proof}

By a similar argument, the proposition above also
holds replacing $H$ by $H^c$. Also, it is easy to see that if $E$
satisfies \eqref{sup}, then $E^c$ satisfies \eqref{sub}, hence
by applying Proposition~\ref{comp} to $E^c$ and $H^c$
one has the following corollary.

\begin{prop}[Comparison principle II]\label{comp2}
Let $E$ satisfy \eqref{sup} with $\Omega=Q$ and assume that $E\cap
Q^c\subset H$. Then $E\subset H$ up to a set of measure zero (i.e.
$|H^c\cap E|=0$).
\end{prop}

\medskip

\noindent\emph{Proof of Proposition \ref{Hmin}.} According to
Proposition~\ref{exist} a minimizer $E$ of $\M{J}_s(\cdot,Q)$ with
$E\cap Q^c=H\cap Q^c$ exists. Then $E$ satisfies both \eqref{sub}
and \eqref{sup}, hence by Propositions \ref{comp} and \ref{comp2} we
have $H\subset E$ and $E\subset H$ (up to sets of measure $0$), i.e.
$E=H$. \hfill $\square$

\subsection{Monotone set functions}\label{setfunctions}

We report some of the main results of \cite{DGL}, see also
\cite[Chapter~16]{DM} for more general and related results.
In the sequel for an open set $ \Omega\subset\R{n} $, we denote by $ \mathcal P (\Omega) $ the
set of subsets of $ \Omega $ and by $ \M{A}(\Omega),\
\M{K}(\Omega)\subset  \M{P} (\Omega) $, the collection of open and
compact subset of $ \Omega $ respectively. We also define
\[
\M{C}(\Omega):=\Big\{ \bigcup_{i=1}^{M} Q_i:\ Q_i\in \M{Q},\ M \in \mathbb{N}\Big\},
\]
where $ \M{Q} $ is \emph{countable} the set of \emph{open} cubes $ Q_r(x):= x+rQ \Subset \Omega $ with $x\in \mathbb Q^n  $ and $0< r\in \mathbb Q $. The collections $\M{A}(\Omega)$, $\M{K}(\Omega)$ and $\M{C}(\Omega)$ satisfy the following property
\begin{equation}\label{inclusioni}
A\in \M{A}(\Omega),\ K\in\M{K}(\Omega),\ K\subset A\ \Rightarrow \ \text{there exists } C\in \M{C}(\Omega) \text{ with } K\subset C\Subset A.
\end{equation}
We say that a set function $ \alpha : \M{P}(\Omega) \to [0,\infty] $ is   \emph{monotone} if
\[ \alpha(E)\le \alpha (F)  \text{ wherever $ E \subset F $},
\]
and that a monotone set function is \emph{regular} if the following two conditions hold
\begin{eqnarray}
\alpha(A)&=&\sup\{\alpha(K):\ K\subset A,\ K \in \M{K}(\Omega)\}\text{ for any $ A \in \M{A}(\Omega), $}\label{reg-}\\
\alpha(E)&=&\inf\{\alpha(A):\ E\subset A,\ A\in \M{A}(\Omega)\}\text{ for any $ E \in  \M{P}(\Omega). $ }\label{reg+}
\end{eqnarray}
Thanks to \eqref{inclusioni} it is clear that \eqref{reg-} is equivalent to
\begin{equation}\label{equiv}
\alpha(A) =\sup\{\alpha(V):\ V\Subset A,\ V \in \M{A}(\Omega)\}
=\sup\{\alpha(C):\ C\Subset A,\ C \in \M{C}(\Omega)\}.
\end{equation}
We also say that a monotone set function $\alpha$ is \emph{super-additive}  if
\[
\alpha(E\cup F)\ge \alpha (E)+\alpha(F),\quad \text{wherever } E,F\in \M{P}(\Omega),\ E\cap F=\emptyset.
\]

We say that a sequence  of regular monotone set functions $\alpha_i$ \emph{weakly converges} to a monotone set function $ \alpha $ if the following two conditions hold:
\begin{eqnarray}
\liminf_{i \to \infty} \alpha_i(A)&\ge& \alpha(A) \text{ for every $ A\in \M{A}(\Omega), $}\label{lim-}\\
\limsup_{i \to \infty} \alpha_i(K)&\le& \alpha(K) \text{ for every $ K\in \M{K}(\Omega) $}.\label{lim+}
\end{eqnarray}
The limit need not be unique, but it is easy to see that a sequence of regular monotone set functions admits at most one \emph{regular} limit.

\begin{trm}[De Giorgi-Letta]\label{DGL} Let $ (\alpha_i) $ be a sequence of regular monotone set functions such that
\[
\limsup_{i\to\infty} \alpha_i(\Omega')<\infty\quad \text{for every open set }\Omega'\Subset\Omega.
\]
Then there exists a subsequence $ (\alpha_{i'}) $  weakly converging to a regular monotone set function $ \alpha $. Moreover if each $ \alpha_i $ is super-additive on disjoint open sets\footnote{This means that $ \alpha_i(A\cup B) \ge \alpha_i(A)+\alpha_i(B) $ wherever $ A,B \in \M{A}(\Omega) $ are disjoint.} (and hence on disjoint compact sets), then so is $ \alpha $.
\end{trm}
\begin{proof}
Since the proof is standard we only sketch it.

\noindent \emph{Step 1.} Being $\M{C}(\Omega)$ countable, we can easily extract a diagonal subsequence, still denoted by $(\alpha_i) $ such that,
\[
\beta(C) :=  \lim_{i\to\infty} \alpha_{i}(C)<\infty \quad \text{ for any $ C \in \M{C}(\Omega) .$}
\]

\noindent\emph{Step 2.} We define
\[
\begin{split}
\alpha(A)&:=\sup\big\{\beta(C):\ C \Subset A,\ C\in \M{C}(\Omega)\big\} \quad \text{for every }A \in \M{A}(\Omega),\\
\alpha(E)&:=\inf\big\{\alpha(A):\ A\supset E,\  A \in \M{A}(\Omega)\big\}\quad \text{for every }E \in\M{P}(\Omega).
\end{split}
\]
Clearly for $ C \in \M{C}(\Omega) $ we have $\alpha(C)\le \beta(C)  $.

\noindent\emph{Step 3.} The set function $ \alpha $ is clearly monotone, and if every $\alpha_{i}$ is
super-additive on disjoint open sets, then so is $\alpha$. It is also easy to see that \eqref{lim-} is satisfied.
As for \eqref{lim+}, it is an easy consequence of the identity
\[
\alpha(K)=\inf\{\beta(C):\ C\supset K,\ C\in \M{C}(\Omega)\}.
\]
which follows from \eqref{inclusioni}. Then $\alpha_i$ converges weakly to $\alpha$.

\noindent\emph{Step 4.}
It remains to prove the regularity of $ \alpha $. Identity \eqref{reg+} follows by the definition of $\alpha$.
In order to prove \eqref{reg-} fix any $A\in \M{A}(\Omega)$. Then for $C\in \M{C}(\Omega)$ with $C\Subset A$,
we have
$$\beta(C)=\lim_{i\to\infty}\alpha_i(C)\le \limsup_{i\to\infty}\alpha_i(\overline C)\le
\alpha(\overline C)\le\alpha (C')\le \beta (C').$$
 From this and the definition of $\alpha(A)$,
\eqref{equiv} follows at once, hence also \eqref{reg-}.
\end{proof}

\begin{prop}\label{useful}Let $ (\alpha_i) $ be a sequence of regular monotone set functions weakly converging to a regular monotone set function $ \alpha $, and let $K_j\downarrow K  $ be a decreasing sequence of compact sets such that
$ \alpha(K)=0 $. Then
\[
\lim_{j\to \infty} \limsup_{i\to \infty} \alpha_i(K_j)=0
\]
\end{prop}
\begin{proof} We have
\[
0=\alpha(K)=\lim_{j\to \infty} \alpha(K_j)\ge \lim_{j\to \infty}\limsup_{i\to \infty} \alpha_i(K_j),
\]
where the second equality follows from the regularity of $ \alpha $. Indeed for $A\in\M{A}(\Omega)$ with $A\supset K$, we have by compactness $A\supset K_j$ for $j$ large enough, hence
$$\alpha(A)\ge \lim_{j\to\infty} \alpha(K_j)\ge \alpha(K)=0,$$
and the claim follows by taking the infimum over all $A\in \M{A}(\Omega)$ with $A\supset K$.
\end{proof}

\end{document}